\documentclass[12pt]{article}

\usepackage{amsmath,amsfonts,amssymb}
\usepackage{epsf}
\RequirePackage{amsfonts,amssymb,amsmath,amscd,amsthm}
\usepackage{epsfig,tabularx,graphicx,float}
\usepackage{graphicx}
\usepackage[a4paper]{geometry}

\usepackage{parskip}
\usepackage[utf8]{inputenc}
\usepackage{times}
\usepackage{float}
\usepackage{multicol}
\usepackage{multirow}

\usepackage{float}
\newtheorem{thm}{Theorem}[section]

\newtheorem{defi}[thm]{Definition}

\newtheorem{lem}[thm]{Lemma}
\begin{document}

	\begin{center}
		\textbf{Characterization of the pseudo-scaling functions on Vilenkin group}\\[1 cm]
		
		\text{Prasadini Mahapatra}\\

	\end{center}
	\begin{abstract}
	The study of wavelets, originated from it’s applications in diverse fields, was combined together by it’s Mathematical properties. Initially, all the wavelets and it’s variants were explored in the real space $\mathbb{R}^{n}$. But now, these are being studied in different abstract settings. The present paper also contributes to this extension. Vilenkin groups, introduced by F. Ya Vilenkin, form a class of locally compact abelian groups. In this paper, Parseval frame multiwavelets associated to multiresolution analysis (MRA) are characterized in $L^{2}(G)$, where $G$ is the Vilenkin group. Further, we introduce the pseudo-scaling function along with a class of generalized low pass filters and study their properties in Vilenkin group.	
	\end{abstract}
	
	\section{Introduction}
	
	During the last two decades, several authors studied more generalizations and extensions of wavelets. Multiresolution analysis (MRA) is very fundamental tool to establishment of scaling function, which appeared in very different contexts. This paper is related to one such generalization of wavelets.

	Walsh functions were introduced by J. Walsh     In 1923, J. Walsh  introduced that the linear combination of Haar functions is known as Walsh functions. At first, N. J. Fine and N. Ya Vilenkin independently determined that Walsh system is the group of characters of the Cantor dyadic group. A large class of locally compact abelian groups, called Vilenkin groups is introduced by Vilenkin. Cantor dyadic group is a particular case. Refinable equation gives refinable function which generates MRA and hence wavelets, if the mask satisfies certain conditions. Necessary and sufficient conditions were given over the mask of scaling function $\phi$ in terms of modified Cohen's condition and blocked sets such that $\phi$ generates an MRA.  
	
	Wavelets and multiwavelets related work on Vilenkin group has been done in several paper. In case of Vilenkin group if the associated prime $p$ is greater than 2, then MRA generates a multiwavelet set having $p-1$ functions. 

	In section 2, It has been introduced that the pseudo scaling functions associated with the filters and generalized filters. Then we characterize the generalized low pass filters, which motivate us to construct the subclass
	of MRA Parseval frame multiwavelets. Furthermore, we have studied that the associated class of pseudo-scaling functions which are not necessarily obtained from a multiresolution analysis.

	\subsection{Preliminaries}
	 Vilenkin group $ G $ is defined as the group of sequences
	$$
	x=(x_{j})=(..., 0, 0, x_{k}, x_{k+1}, x_{k+2},...),
	$$
	where $x_{j}\in \lbrace 0, 1, ..., p-1\rbrace$, $p$ is prime, for $ j \in \mathbb{Z} $ and $ x_{j}=0 $, for $ j < k = k(x)$. The group operation on $ G $, denoted by $ \oplus $, is defined as coordinatewise addition modulo $ p $:
	$$
	(z_{j})=(x_{j})\oplus (y_{j})\Leftrightarrow z_{j}=x_{j} + y_{j}(\text{mod}\: p), \text{ for } j \in \mathbb{Z}.
	$$
	$\theta$ denotes the identity element (zero) of $G$.
	Let
	$$
	U_{l}=\lbrace(x_{j}) \in G : x_{j}=0\: \text{for}\: j\leq l\rbrace ,\qquad l\in \mathbb{Z},
	$$
	be a system of neighbourhoods of zero in $G$. In case of topological groups if we know neighbourhood system $\lbrace U_l \rbrace_{l \in \mathbb{Z}}$ of zero, then we can determine neighbourhood system of every point $x=(x_j) \in G$ given by $\lbrace U_l \oplus x \rbrace_{l \in \mathbb Z}$, which in turn generates a topology on $G$.

	Let $ U=U_{0} $ and $ \ominus $ denotes the inverse operation of $ \oplus $. The Lebesgue spaces $ L^{q}(G),\: 1\leq q\leq \infty $, are defined with respect to the Haar measure $ \mu $ on Borel subsets of $ G $ normalized by $ \mu(U)=1 $.

	The group dual to $ G $ is denoted by $ G^{*} $ and consists of all sequences of the form
	$$
	\omega=(\omega_{j})=(...,0,0,\omega_{k},\omega_{k+1},\omega_{k+2},...),
	$$
	where $ \omega_{j} \in \lbrace 0,1,...,p-1\rbrace $, for $ j \in \mathbb{Z} $ and $ \omega_j=0 $, for $ j<k=k(\omega) $. The operations of addition and subtraction, the neighbourhoods $\lbrace U_{l}^{*}\rbrace $ and the Haar measure $ \mu^{*} $ for $ G^{*} $ are defined as above for $ G $. Each character on $ G $ is defined as
	$$
	\chi(x,\omega)=\exp\bigg(\frac{2\pi i}{p}\sum_{j\in \mathbb Z}{x_{j}w_{1-j}}\bigg),\quad x \in G,
	$$
	for some $ \omega \in G^{*} $.

	Let $ H = \lbrace (x_{j}) \in G\: |\: x_{j} =0, \; \text{ for } j>0\rbrace $ be a discrete subgroup in $ G $  and $A$ be an automorphism on $G$ defined by $ (Ax)_j=x_{j+1} $, for $x=(x_j) \in G$. From the definition of annihilator and above definition of character $\chi$, it follows that the annihilator $ H^{\perp} $ of the subgroup $ H $ consists of all sequences $ (\omega_j) \in G^{*} $ which satisfy $ \omega_j=0 $ for $ j>0 $.

	Let $ \lambda:G\longrightarrow \mathbb R_+ $ be defined by
	\begin{equation*}
	\lambda(x)=\sum_{j \in \mathbb Z}{x_{j} p^{-j}}, \qquad x=(x_j) \in G.
	\end{equation*}
	It is obvious that the image of $ H $ under $ \lambda $ is the set of non-negative integers $ \mathbb Z_{+} $. For every $ \alpha \in \mathbb Z_{+} $, let $ h_{[\alpha]} $ denote the element of $ H $ such that $ \lambda(h_{[\alpha]})=\alpha $. For $ G^{*} $, the map $ \lambda^{*}:G^{*}\longrightarrow \mathbb R_{+} $, the automorphism $ B \in \text{Aut } G^{*} $, the subgroup $ U^{*} $ and the elements $ \omega_{[\alpha]} $ of $ H^{\perp} $ are defined similar to $ \lambda$, $A$, $U$ and $h_{[\alpha]}$, respectively. 

	The generalised Walsh functions for $ G $ are defined by
	\begin{equation*}
	W_\alpha(x)=\chi(x,\omega_{[\alpha]}),\qquad \alpha \in \mathbb Z_+, x \in G.
	\end{equation*}
	These functions form an orthogonal set for $L^2(U)$, that is,
	\begin{equation*}
	\int_{U}{W_{\alpha}(x) \overline{W_{\beta}(x)}d\mu(x)}=\delta_{\alpha,\beta},\qquad \alpha,\beta \in \mathbb Z_{+},
	\end{equation*}
	where $ \delta_{\alpha,\beta} $ is the Kronecker delta. The system $ {W_{\alpha}} $ is complete in $ L^{2}(U) $. The corresponding system for $ G^{*} $ is defined by
	\begin{equation*}
	W_{\alpha}^{*}(\omega)=\chi(h_{[\alpha]},\omega),\qquad \alpha \in \mathbb Z_{+}, \omega \in G^{*}.
	\end{equation*}
	The system $\lbrace W_{\alpha}^{*} \rbrace$ is an orthonormal basis of $ L^{2}(U^{*}) $.

	For positive integers $ n $ and $ \alpha $,
	\begin{equation*}
	U_{n,\alpha}=A^{-n}(h_{[\alpha]}) \oplus A^{-n}(U).
	\end{equation*}


	\subsection{Wavelets on Vilenkin group}
	
	 In \cite{farkov3}, Farkov considered the Strang-fix condition, partition of unit property and the stability of scaling functions on Vilenkin group. Necessary and sufficient conditions are given for scaling functions to generate an MRA in the $L^2$ space on Vilenkin groups by using modified Cohen's condition and blocked sets.

	\subsubsection{Refinable function and mask}
	
	\begin{defi}
		
		Let $ L_{c}^{2}(G) $ be the set of all compactly supported functions in $ L^{2}(G) $. A function $ \phi \in L_{c}^{2}(G) $ is said to be a \textit{refinable function}, if it satisfies an equation of the type
		\begin{equation}
		\phi(x)=p\sum_{\alpha =0}^{p^{n}-1}{a_{\alpha}\phi (Ax\ominus h_{[\alpha ]})}.
		\end{equation}
		The above functional equation is called the \textit{refinement equation}. The generalized Walsh polynomial
		\begin{equation}
		m(\omega)=\sum_{\alpha =0}^{p^{n}-1}{a_{\alpha} \overline{W_{\alpha}^{*}{(\omega)}}}
		\end{equation}
		is called the \textit{mask} of the refinement equation (or the mask of its solution $ \phi $).
	\end{defi}

	\begin{thm}
		\textit{Let $ \phi \in L_{c}^{2}{(G)} $ be a solution of the refinement equation, and let $ \widehat{\phi}(\theta)=1 $. Then,
			$$
			\sum_{\alpha =0}^{p^{n}-1}{a_{\alpha}}=1, \qquad \text{supp } \phi \subset U_{1-n},
			$$
			and
			$$ 
			\widehat{\phi}(\omega)=\prod_{j=1}^{\infty}{m(B^{-j}\omega)}.
			$$
			Moreover, the following properties are true:\\
			$1$. $ \widehat{\phi}(h^{*})=0 $, for all $ h^{*} \in H^{\perp}\setminus {\lbrace \theta\rbrace} $ (the modified Strang-Fix condition),\\
			$2$. $ \sum_{h \in H}{\phi (x\oplus h)}=1 $, for almost every $ x \in G $ (the partition of unit property)}.
	\end{thm}

	\begin{defi}
		A set $ M\subset U^* $ is said to be \textit{blocked} (for the mask $ m $) if it coincides with some union of the sets $ U_{n-1,s}^{*}$, $0\leq s\leq p^{n-1}-1 $, does not contain the set $ U_{n-1,0}^{*} $, and satisfies the condition
		$$
		T_{p}{M} \subset M \cup \lbrace \omega \in U^{*} : m(\omega)=0\rbrace,
		$$
		where
		$$
		T_pM=\bigcup_{l=0}^{p-1} \lbrace B^{-l}\omega_{[l]}+B^{-1}(\omega): \omega \in M \rbrace.
		$$
	\end{defi}

	\subsubsection{Multiresolution analysis}

	\begin{defi}
		A collection of closed subspaces $ V_{j} \subset L^{2}{(G)}, j \in \mathbb Z $, is called a \textit{Multiresolution analysis} (MRA) in $ L^{2}{(G)} $ if the following hold:
		\\
		(i) $ V_{j} \subset V_{j+1} $, for all $ j \in \mathbb Z $
		\\
		(ii) $ \overline{\cup_{j \in \mathbb Z} V_{j}}=L^{2}{(G)} $ and $ \cap_{j \in \mathbb Z} V_{j}=\lbrace 0\rbrace $
		\\
		(iii) $ f(\cdot) \in V_{j}\Leftrightarrow f(A\cdot) \in V_{j+1} $, for all $ j \in \mathbb Z $
		\\
		(iv) $ f(\cdot) \in V_{0}\Rightarrow f(\cdot \ominus h) \in V_{0} $, for all $ h \in H $
		\\
		(v) there is a function $ \phi \in L^{2}{(G)} $ such that the system $ \lbrace \phi(\cdot \ominus h)\vert h \in H\rbrace $ is an orthonormal basis of $ V_{0} $.
		
		The function $ \phi $ in condition (v) is called a \textit{scaling function} of the MRA $(V_j)_{j \in \mathbb Z}$.
	\end{defi}

	For $ \phi \in L^{2}{(G)} $,
	$$
	\phi_{j,h}(x)=p^{j/2}{\phi (A^{j}{x}\ominus h)}, \qquad j \in \mathbb Z, \;\; h \in H
	$$
	and the system $\lbrace \phi_{j,h} : h \in H\rbrace$ forms an orthonormal basis of $ V_{j} $, for every $ j \in \mathbb Z $.

	\begin{thm}
		A function $\phi \in L^{2}(G^{*})$ is a scaling function for an MRA of $L^{2}(G^{*})$ if and only if
		\begin{enumerate}
			\item $\sum_{h \in H^{\perp}}|\hat{\phi}(\omega \oplus h)|^2=1$, \quad \quad for a.e $\omega \in G^{*}$
			\item $lim_{j \rightarrow \infty}|\hat{\phi}(B^{-j}\omega)|=1$, \quad \quad  for a.e $\omega \in G^{*}$
			\item $\hat{\phi}(B\omega)=m(\omega)\hat{\phi}(\omega)$, \quad \quad  for a.e $\omega \in G^{*}$.
		\end{enumerate}
	\end{thm}

	A function $ \phi $ is said to generate an MRA in $ L^{2}{(G)} $ if the system $ \lbrace \phi(\cdot \ominus h)\vert h \in H\rbrace $ is orthonormal in $ L^{2}{(G)} $ and, the family of subspaces
	$$
	V_{j}=\text{clos}_{L^{2}{(G)}}\text{span}\lbrace \phi_{j,h} : h \in H\rbrace, \qquad j \in \mathbb Z,
	$$
	is an MRA in $ L^{2}{(G)}$ with scaling function $\phi$. Farkov gave the following condition under which a compactly supported function $\phi \in L^2(G)$ generates an MRA in $L^2(G)$.

	\begin{thm}
		\textit{Suppose that the refinement equation possesses  a solution $\phi$ such that $\widehat{\phi}(\theta)=1$ and the corresponding mask $m$ satisfies the conditions
			$$
			m(\theta)=1 \qquad \Sigma_{l=0}^{p-1} \vert m(\omega \oplus \delta_{l}) \vert^2=1, \;\; \omega \in G^*,
			$$
			where $\delta_l$ is the sequence $\omega=(\omega_j)$ such that $\omega_1=l$ and $\omega_j=0$ for $j \neq 1$. Then the following are equivalent:\\
			(a) $\phi$ generates an MRA in $L^2(G)$. \\
			(b) $m$ satisfies the modified Cohen's condition.\\
			(c) $m$ has no blocked sets.}
		
	\end{thm}

	Using the above characterization of refinable function and the matrix extension method Farkov gave a procedure for construction of orthonormal wavelets $ \psi_{1},...,\psi_{p-1} $ such that the functions
	$$
	\psi_{l,j,h}(x)= p^{j/2}{\psi_{l}{(A^{j}x\ominus h)}}, \qquad 1\leq l\leq p-1, j \in \mathbb Z, \; \; h \in H,
	$$
	form an orthonormal basis of $ L^{2}{(G)} $.



	
    \section{Characterization the pseudo-scaling functions}
    In this section, We establish the notion of pseudo scaling function and its impact on generalized filters.
    
    For multiwavelet $\Psi:=\{\psi_{1},\psi_{2},...,\psi_{p-1}\}\subset L^{2}(G)$, The affine system $\mathcal{A}(\Psi)$ is defined by
    \begin{equation}\label{11}
    \mathcal{A}(\Psi)=\{\psi_{j,h}^{l}(x)|\psi_{j,h}^{l}(x)=p^{j/2}\psi^{l}(A^{j}x - h):j \in \mathbb{Z}, h \in H, l=1,2,...,p-1\}.
    \end{equation}
    
    The following are the two definitions of multiwavelet frame and the multiwavelet Parseval frame.
    
    \begin{defi}
    	The affine system $\mathcal{A}\subset L^{2}(G)$ is said to be multiwavelet frame if the system \eqref{11} is a frame for $L^{2}(G)$.
    \end{defi} 
    
     \begin{defi}
     	The affine system $\mathcal{A}\subset L^{2}(G)$ is said to be multiwavelet Parseval frame if the system \eqref{11} is a Parseval frame for $L^{2}(G)$.
      \end{defi}

    The following theorem is the characterization of Parseval frame for Vilenkin group, which is one of the particular case of local fields with positive characteristics. 
    
    \begin{thm}
    	Suppose $\Psi=\{\psi_{1},\psi_{2},...,\psi_{p-1}\} \subset L^{2}(G)$. Then the affine system $A(\Psi)$ is a Parseval frame for $L^{2}(G)$ if and only if for a.e. $\omega$, the following holds:
    	\begin{itemize}
    		\item [1.]$\sum_{i=1}^{p-1}\sum_{j \in \mathbb{Z}}|\hat{\psi}_{i}(B^{j}\omega)|^{2}=1$
    		\item [2.] $\sum_{i=1}^{p-1}\sum_{j=0}^{\infty}\hat{\psi}_{i}(B^{j}\omega)\overline{\hat{\psi}_{i}(B^{j}(\omega+\gamma))}=0$, for $\gamma \in H^{\perp}\setminus BH^{\perp}$.
    	\end{itemize} 
    \end{thm}
    The proof of this theorem are in \cite{behera}.
    \begin{defi} 	
    	Let  $M=\{m_{0},m_{1},...,m_{p-1}\} \subset L^{\infty}(U)$ be a $H-$ periodic function, is called a generalized filter if it satisfies the following equation
    \begin{equation}\label{12}
    \sum_{i=0}^{p-1}|m_{i}(\omega)|^{2}=1, \text{ a.e } \omega,
    \end{equation}	
    \begin{equation}\label{13}
    m_{0}(\omega)\overline{m_{0}(\omega +\beta)}-(\sum_{i=1}^{p-1}m_{i}(\omega)\overline{m_{i}(\omega + \beta)})=0 \text{ a.e }\omega, \text{ where } \beta=B^{-1}\gamma \text{ and }\gamma \text{  as in Theorem 2.3 }.
    \end{equation} 	
    \end{defi}
     We define the set of generalized filters is denoted by $\tilde{F}$ and let $\tilde{F}^{+}=\{M \in \tilde{F}:m_{0} \geq 0, m_{0} \in M\}$. Notice that for $M \in \tilde{F}$,   $M_{|m_{0}|}=\{|m_{0}|,m_{1},...,m_{p-1}\} \in \tilde{F}^{+}$.
   
   \begin{defi}
   	A function $\phi \in L^{2}(G)$ is called a pseudo-scaling function if there exists a generalized filter $M=\{m_{0},m_{1},...,m_{p-1}\} \in \tilde{F}$ such that 
   	\begin{equation}\label{14}
   	\hat{\phi}(B\omega)=m_{0}(\omega)\hat{\phi}(\omega), \text{ a.e } \omega.
   	\end{equation}
   	Here we observed that $M$ is not uniquely determined by the pseudo-scaling function $\phi$. Therefore, we consider the set of all $M \in \tilde{F}$ such that 
   	$M$ satisfies \eqref{14} for $\phi$, and is denoted by $\tilde{F}_{\phi}$. In particular, if $\phi=0$, then $\tilde{F}_{\phi}=\tilde{F}$.
   	
   	If $M \in \tilde{F}^{+}$, then 
   	$$\hat{\phi}_{m_{0}}(\omega)=\prod_{j=1}^{\infty}m_{0}(B^{-j}\omega) \quad \text{ is well defined a.e. },$$
   	since $0 \leq m_{0}(\omega) \leq 1$, a.e. $\omega$.
   	Furthermore, we get
   	\begin{equation}\label{15}
   		\hat{\phi}_{m_{0}}(B\omega)=m_{0}(\omega)\hat{\phi}_{m_{0}}(\omega), \text{ a.e } \omega.
   	\end{equation}
   \end{defi} 
   \begin{defi}
   	Suppose that $M=\{m_{0},m_{1},...,m_{p-1}\} \in \tilde{F}_{\phi_{m_{0}}}$.  Let us define 
   		\begin{equation}\label{16}
   		N_{0}(m_{0})=\{\omega \in G^{*}: lim_{j \rightarrow \infty}\hat{\phi}_{m_{0}}(B^{-j}\omega)=0\}.
   		\end{equation}
   	  If $|N_{0}(|m_{0}|)|=0$, then $M$is called a generalized low-pass filter. 
   	The set denoted by $\tilde{F}_{0}$ is the set of all generalized filters satisfying equation \eqref{16}.			
   \end{defi} 
   
    \begin{lem}\label{21}
    	If $f \in L^{2}(G^{*})$, then $lim_{j \rightarrow \infty}|f(B^{j}\omega)|=0$,  for a.e $\omega \in G^{*}$. 
    \end{lem}
 \begin{proof}
  Let us assume that $f \in L^{2}(G^{*})$ and by applying the monotone convergence theorem we get
  \begin{align*}
  \int_{G^{*}}\sum_{j \in \mathbb{Z}^{+}}|f(B^{j}\omega)|^{2}d\omega &= \sum_{j \in \mathbb{Z}^{+}}\int_{G^{*}}|f(B^{j}\omega)|^{2}d\omega\\
  &=\sum_{j \in \mathbb{Z}^{+}}p^{-j}\int_{G^{*}}|f(\omega)|^{2}d\omega\\
  &=\frac{1}{p-1}||f||^{2} <\infty.
  \end{align*}
   That implies $\sum_{j \in \mathbb{Z}^{+}}|f(B^{j}\omega)|^{2}$ is finite, for $\omega \in G^{*}$ a.e. Thus, for a.e. $\omega \in G^{*}$, $lim_{j \rightarrow \infty}|f(B^{j}\omega)|=0$. 	
  \end{proof}  
       
  \begin{defi}
  	A Parseval frame  multiwavelet(PFMW) $\Psi=\{\psi_{1},\psi_{2},...,\psi_{p-1}\}$ is called an MRA PFMW if there exists a pseudo-scaling function $\phi$, $M \in \tilde{F}_{\phi}$ and unimodular functions $s_{i}\in L^{2}(G)$, $1 \leq i \leq p-1$ such that
  	\begin{equation}\label{17}
  	\hat{\psi}_{i}(B\omega)=W_{\alpha}(\omega)s_{i}(B\omega)\overline{M(\omega)}\hat{\phi}(\omega), \text{ a.e } \omega.
  	\end{equation}	
  \end{defi}  
 The following theorem gives a characterization of the generalized low pass filter.
 \begin{thm}
 	Suppose $\Psi=\{\psi_{1},\psi_{2},...,\psi_{p-1}\}$ is an MRA PFMW and $\phi$ is a pseudo-scaling function satisfying \eqref{14}. If M is defined by \eqref{17},
 	then $M \in \tilde{F}_{0}$.
 \end{thm}
 \begin{proof}
 	Notice that $\Psi$ is an MRA PFMW, by Theorem 2.3, \eqref{12}, \eqref{13} and \eqref{17}, we have 
 	\begin{align*}
 	1 &= \sum_{i=1}^{p-1}\sum_{j \in \mathbb{Z}}|\hat{\psi}_{i}(B^{j}\omega)|^{2}\\
 	&=\sum_{i=1}^{p-1}\sum_{j \in \mathbb{Z}}|W_{\alpha}(\omega)s_{i}(B^{j}\omega)\overline{M(B^{j-1}\omega)}|^{2}|\hat{\phi}(B^{j-1}\omega)|^{2}\\
 	&=\sum_{i=1}^{p-1}\sum_{j \in \mathbb{Z}}|m_{i}(B^{j-1}\omega)|^{2}|\hat{\phi}(B^{j-1}\omega)|^{2}\\
 	&=lim_{n \rightarrow \infty}\sum_{j=-n}^{n}(\sum_{i=1}^{p-1}|m_{i}(B^{j-1}\omega)|^{2})|\hat{\phi}(B^{j-1}\omega)|^{2}\\
 	&=lim_{n \rightarrow \infty}\sum_{j=-n}^{n}(1-|m_{0}(B^{j-1}\omega)|^{2})|\hat{\phi}(B^{j-1}\omega)|^{2}\\
 	&=lim_{n \rightarrow \infty}\{|\hat{\phi}(B^{-n-1}\omega)|^{2}-|\hat{\phi}(B^{n}\omega)|^{2}\}.
 	\end{align*}
 	By using the Lemma $\eqref{21}$ for  $\phi \in L^{2}(G)$, we get  $lim_{n \rightarrow \infty}|\hat{\phi}(B^{n}\omega)|^{2}=0$ for a.e. $\omega$.
 	Therefore, $lim_{n \rightarrow \infty}|\hat{\phi}(B^{-n-1}\omega)|^{2}=1$ holds for a.e $\omega$.
 	From \eqref{14}, we have 
 	$$|\hat{\phi}(\omega)|=|\prod_{j=1}^{n}M(B^{-j}\omega)||\hat{\phi}(B^{-j}\omega)|, \text{ a.e. } \omega.$$ 
 	Using \eqref{15}, It is obtained that $|\hat{\phi}(\omega)|=\hat{\phi}_{|M|}$ and $|N_{0}(|M|)|=0$
 	is clearly satisfied. Thus, by Definition 2.5, we have $M \in \tilde{F}_{0}$. 
 \end{proof}
\begin{lem}\label{22}
	Let $\mu$ be a $H-$ periodic, unimodular function. Then there exists a unimodular function $v$ that satisfies 
	\begin{equation}\label{18}
	\mu(\omega)=v(B\omega)\overline{v(\omega)},  \quad \text{ a.e. }\omega.
	\end{equation}	
\end{lem} 
 
 \begin{thm}
 	Let $M \in \tilde{F}_{0}$ be a generalizedfilter. Then there exist a pseudo-scaling function and an MRA PFMW $\Psi=\{\psi_{1},\psi_{2},...,\psi_{p-1}\}$ such that they satisfy \eqref{17}.
 \end{thm}
 \begin{proof}
 	Suppose that $M \in \tilde{F}_{0}$ a generalized filter. For $m_{0}\in M$, define the signum function $\mu$ such that
 	$$\mu(\omega)=\left\{\begin{matrix}
                 \frac{m_{0}(\omega)}{|m_{0}(\omega)|}& m_{0}(\omega)\neq 0;\\ 
 	             1 & m_{0}(\omega)=0. 
              	\end{matrix}\right. $$
  Clearly, $\mu$ is a measurable unimodular function and we see the following equation holds for a.e. $\omega$           	
   $$ m_{0}(\omega)=\mu(\omega)|m_{0} (\omega)|.$$ 
   By Lemma \eqref{22}, there exists a unimodular measurable function $v$ such that 
   $$\mu(\omega)=v(B\omega)\overline{v(\omega)},  \quad \text{ a.e. }\omega.$$	
    By \eqref{15}, we have the function $\hat{\phi}_{|m_{0}|}(\omega)$ from $|m_{0}|$. Let          	         	
    \begin{equation}\label{19}
     \hat{\phi}(\omega)=v(\omega)\hat{\phi_{|m_{0}|}}(\omega),
    \end{equation}
    then, $\phi \in L^{2}(G)$. Using \eqref{15},\eqref{18}, \eqref{19} and the definition of the signum function $\mu$, we have
    \begin{align*}
    \hat{\phi}(B\omega) &=v(B\omega)\hat{\phi_{|m_{0}|}}(\omega)\\
    &=v(B\omega)|m_{0}(\omega)||\hat{\phi_{|m_{0}|}}(\omega)\\
    &==v(B\omega)|m_{0}(\omega)|\overline{v(\omega)}\hat{\phi}(\omega)\\
    &=\mu(\omega)|m_{0}(\omega)|\hat{\phi}(\omega)\\
    &=m_{0}(\omega)\hat{\phi}(\omega).
    \end{align*}
  Thus, $\phi$ is a pseudo-scaling function.  
  For $\hat{\Psi}=\{\psi_{1},\psi_{2},...,\psi_{p-1}\}$, let $\hat{\psi}_{i}=W_{\alpha}(\omega)s_{i}(B\omega)\overline{M(\omega)}\hat{\phi}(\omega)$, $1 \leq i \leq p-1$, a.e. $\omega$,
 where $s_{i}(\omega) \in L^{2}(G),m_{i}\in M$. By Theorem 2.3, we see that $\Psi$ is an MRA PFMW.
 Note that $m_{0}$ is a generalized low pass filter, that implies $lim_{j \rightarrow \infty}|\hat{\phi}(B^{-j}\omega)|=1$, a.e. $\omega$ holds. By using Lemma \eqref{21}, we get
 \begin{align*}
 \sum_{i=1}^{p-1}\sum_{j \in \mathbb{Z}}|\hat{\psi}_{i}(B^{j}\omega)|^{2} 
 &= \sum_{i=1}^{p-1}\sum_{j \in \mathbb{Z}}|m_{i}(B^{j-1}\omega)|^{2}|\hat{\phi}_{i}(B^{j-1}\omega)|^{2}\\
 &=lim_{n \rightarrow \infty}\sum_{j=-n}^{n}(\sum_{i=1}^{p-1}|m_{i}(B^{j-1}\omega)|^{2})|\hat{\phi}(B^{j-1}\omega)|^{2}\\
 &=lim_{n \rightarrow \infty}\sum_{j=-n}^{n}(1-|m_{0}(B^{j-1}\omega)|^{2})|\hat{\phi}(B^{j-1}\omega)|^{2}\\
 &=lim_{n \rightarrow \infty}\{|\hat{\phi}(B^{-n-1}\omega)|^{2}-|\hat{\phi}(B^{n}\omega)|^{2}\}.
 \end{align*}
 Using $\phi \in L^{2}(G)$, Lemma $\eqref{21}$ implies $lim_{n \rightarrow \infty}|\hat{\phi}(B^{n}\omega)|^{2}=0$ for a.e. $\omega$.
 Then for  a.e $\omega$, $lim_{n \rightarrow \infty}|\hat{\phi}(B^{-n-1}\omega)|^{2}=1$ holds.
Now to show that $\Psi$ given above satisfies the second condition of Theorem 2.3.
Fixthe $\omega$ and $q = Bh \oplus \gamma,\text{ where } h \in H, \gamma \in H^{\perp}/BH^{\perp}$ and split the equation as\\
$\sum_{i=1}^{p-1}\sum_{j=0}^{\infty}\hat{\psi}_{i}(B^{j}\omega)\overline{\hat{\psi}_{i}(B^{j}(\omega+q))}$
\begin{equation}\label{20}
=\sum_{i=1}^{p-1}\hat{\psi}_{i}(\omega)\overline{\hat{\psi}_{i}(\omega+q)}+\sum_{i=1}^{p-1}\sum_{j=1}^{\infty}\hat{\psi}_{i}(B^{j}\omega)\overline{\hat{\psi}_{i}(B^{j}(\omega+q)}
\end{equation}
 Using \eqref{13}, \eqref{14} and Lemma
 \eqref{21}, the first term on the right-hand side of \eqref{20} have the following equation.   
 
\begin{align*}
\sum_{i=1}^{p-1}\hat{\psi}_{i}(\omega)\overline{\hat{\psi}_{i}(\omega+q)}\\
&=\sum_{i=1}^{p-1}W_{\alpha}(B^{-1}\omega)s_{i}(\omega)m_{i}(B^{-1}\omega)\hat{\phi}(B^{-1}\omega)\\
&*\overline{W_{\alpha}(B^{-1}(\omega+q))}\overline{s_{i}(\omega+q)m_{i}(B^{-1}\omega+q)\hat{\phi}(B^{-1}\omega+q)}\\
&=\overline{W_{\alpha}(B^{-1}q)}(\sum_{i=1}^{p-1}|s_{i}(\omega)|^{2}m_{i}(B^{-1}\omega)\overline{m_{i}(B^{-1}(\omega+q))})*\hat{\phi}(B^{-1}\omega)\overline{\hat{\phi}(B^{-1}(\omega+q))}\\
&=\overline{W_{\alpha}(B^{-1}q)}\overline{m_{0}(B^{-1}(\omega+B^{-1}q)})*\hat{\phi}(B^{-1}\omega)\overline{\hat{\phi}(B^{-1}(\omega+q))}\\
&=-\hat{\phi}(\omega)\overline{\hat{\phi}(\omega+q)}.
\end{align*}   
 By \eqref{12}, \eqref{14}, \eqref{17} and Lemma \eqref{21}, for the second term on the right-hand side of \eqref{20}, we have   
  \begin{align*}
  \sum_{i=1}^{p-1}\sum_{j=1}^{\infty}\hat{\psi}_{i}(B^{j}\omega)\overline{\hat{\psi}_{i}(B^{j}(\omega+q))}\\
  &=\sum_{i=1}^{p-1}\sum_{j=1}^{\infty}W_{\alpha}(B^{j-1}\omega)s_{i}(B^{j}\omega)m_{i}(B^{j-1}\omega)\hat{\phi}(B^{j-1}\omega)\\
  &*\overline{W_{\alpha}(B^{j-1}(\omega+q))}\overline{s_{i}(B^{j}(\omega+q))m_{i}(B^{j-1}(\omega+q))\hat{\phi}(B^{j-1}(\omega+q))}\\
  &=\sum_{j=1}^{\infty}\overline{W_{\alpha}(B^{j-1}q)}(\sum_{i=1}^{p-1}|s_{i}(B^{j}\omega)|^{2}m_{i}(B^{j-1}\omega)\overline{m_{i}(B^{j-1}(\omega+q))})\\
  &*\hat{\phi}(B^{j-1}\omega)\overline{\hat{\phi}(B^{j-1}(\omega+q))}\\
  &=\sum_{j=1}^{\infty}(1-|m_{0}(B^{j-1}(\omega))|^{2})\hat{\phi}(B^{j-1}\omega)\overline{\hat{\phi}(B^{j-1}(\omega+q))}\\
  &=\sum_{j=1}^{\infty}\{\hat{\phi}(B^{j-1}\omega)\overline{\hat{\phi}(B^{j-1}(\omega+q))}-\hat{\phi}(B^{j}\omega)\overline{\hat{\phi}(B^{j}(\omega+q))}\}\\
  &=\hat{\phi}(\omega)\overline{\hat{\phi}(\omega+q)}-lim_{p \rightarrow \infty}\hat{\phi}(B^{p}\omega)\overline{\hat{\phi}(B^{p}(\omega+q))}\\
  &=\hat{\phi}(\omega)\overline{\hat{\phi}(\omega+q)}.
  \end{align*}            	
   
 Adding both the result that we obtained above, right hand side of \eqref{20} is equal to 0. Hence, $\Psi$ is a PFMW by theorem 2.3.  	
 \end{proof}

\end{document}